\documentclass[12pt]{amsart}

\usepackage{pict2e}
\usepackage{amssymb}
\usepackage{graphicx}
\usepackage{amsmath}
\usepackage{amscd}
\usepackage{latexsym}
\usepackage{tikz}
\usetikzlibrary[shapes,arrows,snakes]
\tikzstyle{bull}=[circle,draw=black,fill=black!80]
\tikzstyle{holl}=[circle,draw=black]
\include{amsfonts}
\newtheorem{thm}{Theorem}[section]

\newtheorem{defn}[thm]{Definition}
\newtheorem{rmk}[thm]{Remark}

\newtheorem{prop}[thm]{Proposition}

\newtheorem{claim}[thm]{Claim}

\newcommand{\ts}{\textsc}

\newcommand{\pushout}[9]{\parbox{2in}{
\setlength{\unitlength}{.7in}
\begin{center}
\begin{picture}(2,2.5)(1,.5)
\put(1.8,2.8){\vector(-1,-1){.6}}
\put(2.8,1.8){\vector(-1,-1){.6}}
\put(2.2,2.8){\vector(1,-1){.6}}
\put(1.2,1.8){\vector(1,-1){.6}}
\put(2,3){\makebox(0,0){$#1$}}              
\put(1.1,2){\makebox(0,0){$#2$}}            
\put(2.9,2){\makebox(0,0){$#3$}}            
\put(2,1){\makebox(0,0){$#4$}}
\put(1.35,2.65){\makebox(0,0)[r]{$#5$}}     
\put(2.65,2.65){\makebox(0,0)[l]{$#6$}}     
\put(1.35,1.35){\makebox(0,0)[r]{$#7$}}
\put(2.65,1.35){\makebox(0,0)[l]{$#8$}}
\end{picture}
\end{center}}}

\begin{document}

\title{Sections in orthomodular structures of decompositions}
\author{{\bf John Harding and Taewon Yang}}
\date{}

\maketitle
\markboth{John Harding and Taewon Yang}{Sections of Decompositions}

\begin{abstract}
There is a family of constructions to produce orthomodular structures from modular lattices, lattices that are $M$ and $M^*$-symmetric, relation algebras, the idempotents of a ring, the direct product decompositions of a set or group or topological space, and from the binary direct product decompositions of an object in a suitable type of category. We show that an interval $[0,a]$ of such an orthomodular structure constructed from $A$ is again an orthomodular structure constructed from some $B$ built from $A$. When $A$ is a modular lattice, this $B$ is an interval of $A$, and when $A$ is an object in a category, this $B$ is a factor of $A$. 
\end{abstract}

\section{Introduction}

The key fact in the quantum logic approach to quantum mechanics \cite{BVN,Mackey} is that the closed subspaces of a Hilbert space form an orthomodular poset (abbrev.: \ts{omp}). Quantum logic formulates a portion of quantum mechanics in terms of arbitrary \ts{omp}s, and either attempts to justify the special role played by the \ts{omp} constructed from a Hilbert space, or to propose alternatives to this \ts{omp}. 

A number of types of \ts{omp} arise from constructions very close to the Hilbert space one. Taking a possibly incomplete inner product space $E$, its splitting subspaces are those ordered pairs of orthogonal subspaces $(S,T)$ where $E=S\oplus T$. The collection of splitting subspaces forms an \ts{omp} \cite{Dvurecenskij}. Moving a step further from Hilbert spaces, for any vector space $V$ the ordered pairs $(S,T)$ of subspaces with $V=S\oplus T$ forms an \ts{omp}. Such \ts{omp}s have been considered by a number of authors \cite{Mushtari,Harding1,Ovchinnikov,Chevalier1}. 

The \ts{omp} constructed from pairs of subspaces of a vector space $V$ can be realized from the perspective of lattice theory. The collection of all subspaces of $V$ forms a modular lattice, and those $(S,T)$ with $V=S\oplus T$ are exactly the complementary pairs in this lattice. This construction can be applied to any bounded modular lattice, and even to any lattice that is both $M$ and $M^*$-symmetric \cite{Mushtari,Harding1,Chevalier2}. 

The closed subspaces of a Hilbert space correspond to orthogonal projections, which are certain idempotents of the endomorphism ring. This can be extended to show that the *-projections of any *-ring form an \ts{omp} \cite{Foulis}, and further, that the idempotents of any ring with unit form an \ts{omp} \cite{Katrnoska}. This construction is closely related to those above. For a vector space $V$, the direct sum decompositions $V=S\oplus T$ correspond to the idempotents of the endomorphism ring of $V$, and for a ring $R$, the idempotents of $R$ correspond to direct sum decompositions of the left $R$-module $R_R$. 

The above constructions all have ties to linear algebra and direct sums. A different perspective, allowing movement to a broader setting, comes from the fact that finite direct sums and finite direct products of vector spaces coincide. In \cite{Harding1} it was shown that the direct product decompositions of any set, group, poset, topological space, and uniform space form an \ts{omp}. The direct product decompositions of a set $X$ correspond to certain ordered pairs of equivalence relations called factor pairs \cite{Burris}, so this construction can be made from the algebra of relations on $X$. This can be extended to show that ordered pairs of permuting and complementary equivalence elements of any relation algebra form an \ts{omp}. In \cite{Harding5}, this construction from decompositions was taken to a categorical setting, where it was shown that the direct product decompositions of any object in an honest category (essentially one where ternary product diagrams form a pushout) form a type of orthomodular structure known as an orthoalgebra (abbrev.:~\ts{oa}). See \cite{Harding1,Harding2,Harding3,Harding4,Harding5,Harding6,Harding7,Harding8} for further details on the orthostructures Fact~$A$. 

In all these constructions, an orthomodular structure we call Fact~$A$ is produced from a structure $A$. It is well known that an interval $[0,a]$ of an \ts{omp} or \ts{oa} naturally forms an \ts{omp} or \ts{oa}. It is the purpose of this note to show that such an interval of Fact~$A$ is given by Fact~$B$ for some structure created from $B$. In the case that Fact~$A$ is constructed from complementary pairs $(x,y)$ of a bounded modular lattice $A$, the interval $[0,(x,y)]$ of Fact~$A$ is isomorphic to Fact~$B$ where $B$ is the interval $[0,x]$ of $A$ considered as a modular lattice. When Fact~$A$ is built from the idempotents $e$ of a ring $R$, then $[0,e]$ is isomorphic to Fact~$B$ where $B=\{x:ex=x=xe\}$. Finally, when $A$ is a set, or object in an honest category, then the interval $[0,[A\simeq B\times C]]$ is isomorphic to Fact~$B$. 

This note is organized in the following fashion. In the second section we give the pertinent definitions. In the third, we provide proofs of our result in the case Fact~$A$ is constructed from a bounded modular lattice or ring. These are short and easy, and the results for the vector space setting, lattices that are $M$ and $M^*$-symmetric, set, and relation algebra setting follow directly, or with minor modifications, from these. In the fourth section we provide the most difficult of the results, that of the the construction applied to an object in an honest category. 

\section{Preliminaries} 

\begin{defn}
An orthocomplemented poset is a bounded poset $P$ equipped with a unary \mbox{operation} $'$ that is period two, order inverting, where each $x,x'$ have only the bounds as lower and upper bounds. An orthocomplemented poset is an \ts{omp} if 
\vspace{1ex}
\begin{enumerate}
\item $x\leq y'$ implies $x,y$ have a least upper bound $x\oplus y$.
\item $x\leq y$ implies $x\oplus(x\oplus y')'=y$.
\end{enumerate}
\end{defn}

\begin{defn}
An orthoalgebra (abrev.: \ts{oa}) is a set with partially defined binary operation $\oplus$ that is commutative and associative and constants $0,1$ such that 
\vspace{1ex}
\begin{enumerate}
\item For each $a$ there is a unique element $a'$ with $a\oplus a'=1$.
\item If $a\oplus a$ is defined, then $a=0$. 
\end{enumerate}
\label{oa}
\end{defn}

In an orthocomplemented poset, we say $a,b$ are orthogonal if $a\leq b'$, hence $b\leq a'$. Each \ts{omp} naturally forms an \ts{oa} under the partial binary operation of orthogonal joins. Conversely, in an \ts{oa} define $a\leq b$ if there is $c$ with $a\oplus c = b$. Then with the obvious operation $'$ this forms an orthocomplemented poset in which $a\oplus b$ is a minimal, but not necessarily least, upper bound of $a,b$. This orthocomplemented poset constructed from an \ts{oa} is an \ts{omp} iff $a\oplus b$ gives least upper bounds for all $a,b$. Importantly, the operation $\oplus$ in an \ts{oa} is cancellative, so if $a\oplus b = a\oplus c$, then $b=c$. For further details on \ts{omp}s and for \ts{oa}s, including the following, see \cite{Ptak,Kalmbach,oa}. 

\begin{prop}
If $P$ is an \ts{omp}, then an interval $a\!\downarrow$ of $P$ forms an \ts{omp} under the induced partial ordering and the orthocomplementation defined by $b^{\#}=a\wedge b'$. 
\end{prop}

\begin{prop}
If $A$ is an \ts{oa}, then an interval $a\!\downarrow$ of $A$ forms an \ts{oa} with constants $0,a$ under the restriction of $\oplus$ to this interval. 
\end{prop}

We turn next to a description of methods to construct \ts{omp}s from various types of structures. See \cite{Harding1} for further details. 

\begin{thm}
For a bounded modular lattice $L$, let $L^{(2)}$ to be the set of all ordered pairs of complementary elements of $L$ and define $\leq$ and $'$ on $L^{(2)}$ as follows:
\vspace{1ex}
\begin{enumerate}
\item $(x_1,x_2)\,\leq\, (y_1,y_2)$ iff $x_1\leq y_1$ and $y_2\leq x_2$, 
\item $(x_1,x_2)' = (x_2,x_1)$. 
\end{enumerate}
\vspace{1ex}
Then $L^{(2)}$ is an \ts{omp}. 
\label{lattice}
\end{thm}

\begin{thm}
For $R$ a ring with unit, let $E(R)$ be its idempotents and define a relation $\leq$ and unary operation $'$ on $E(R)$ as follows: 
\vspace{1ex}
\begin{enumerate}
\item $e\leq f$ iff $ef=e=fe$,
\item $e'=1-e$.
\end{enumerate}
\vspace{1ex}
Then $E(R)$ is an \ts{omp}. 
\label{ring}
\end{thm}

\begin{thm}
For $X$ a set let Fact~$X$ be the set of all ordered pairs of equivalence relations $(\theta_1,\theta_2)$ of $X$ such that $\theta_1\cap \theta_2=\Delta$ and $\theta_1\circ\theta_2 = \nabla$, where $\Delta$ and $\nabla$ are the smallest and largest equivalence relations. Define $\leq$ and $'$ on Fact~$X$ as follows:
\vspace{1ex}
\begin{enumerate}
\item $(\theta_1,\theta_2)\leq(\phi_1,\phi_2)$ iff $\theta_1\subseteq\phi_1$, $\phi_2\subseteq\theta_2$, and all relations involved permute, 
\item $(\theta_1,\theta_2)'=(\theta_2,\theta_1)$. 
\end{enumerate}
\vspace{1ex}
Then Fact~$X$ is an \ts{omp}. 
\label{set}
\end{thm}

There are a number of extensions to these results that we briefly describe. 

\begin{rmk}{\em 
Let $(a,b)$ be an ordered pair of elements in a lattice. We say $(a,b)$ is a modular pair, written $(a,b)M$, if $c\leq b$ implies $c\vee (a\wedge b) = (c\vee a)\wedge b$; and $(a,b)$ is a dual-modular pair, written $(a,b)M^*$, if $b\leq c$ implies $c\wedge(a\vee b)=(c\wedge a)\vee b$. A lattice is $M$-symmetric if $(a,b)M$ implies $(b,a)M$, $M^*$-symmetric if $(a,b)M^*$ implies $(b,a)M^*$, and symmetric if it is both $M$ and $M^*$-symmetric. 
The result in Theorem~\ref{lattice} extends to a bounded symmetric lattice $L$ if $L^{(2)}$ is defined to be all complementary pairs of elements that are both modular and dual-modular pairs. 
}
\end{rmk}

\begin{rmk}{\em 
The result in Theorem~\ref{ring} has extension to more general structures known as orthomodular partial semigroups \cite{Gudder,Harding8}. These are structures with a partially defined multiplication that behaves in a similar way to the multiplication of a ring restricted to pairs of commuting elements. }
\end{rmk}

\begin{rmk}{\em 
The result in Theorem~\ref{set} is formulated in terms of the algebra of relations of a set. This can be generalized in an obvious way to apply to any relation algebra, in the sense of Tarski \cite{Harding1}. }
\label{a}
\end{rmk}

We turn to our final constructions of orthomodular structures, those from the direct product decompositions of an object in a certain type of category. We consider categories with finite products, hence a terminal object $\Omega$, and we use $\tau_A$ for the unique morphism $\tau_A:A\to\Omega$. Define an equivalence relation $\simeq$ on the morphisms in a category by setting $f\simeq g$ if there is an isomorphism $u$ with $u\circ f = g$, and use this to define an equivalence relation $\approx$ on the collection of all finite product diagrams by setting $(f_1,\ldots,f_m)\approx (g_1,\ldots,g_n)$ if $m=n$ and $f_i \simeq g_i$ for each $i=1,\ldots,n$. We let the equivalence class of $\approx$ containing $(f_1,\ldots,f_n)$ be $[\,f_1,\ldots,f_n\,]$ and call this equivalence class an $n$-ary decomposition of $A$. Also, for an $n$-tuple of morphisms with common domain $f_i:A\to A_i$, we use $(f_1,\ldots,f_n)$ for the morphism from $A$ into the product of their codomains. See \cite{Harding5} for further details. 

\begin{defn}
In a category with finite products, a binary product diagram $(f_1,f_2)$ where $f_i:A\to A_i$, is called a disjoint binary product if $(f_1,f_2,\tau_{A_1},\tau_{A_2})$ is a pushout. 
\vspace{1ex}

\begin{center}
\pushout{A}{A_1}{A_2}{\Omega}{f_1}{f_2}{\tau_{A_1}}{\tau_{A_2}}{1}
\end{center}
\vspace{-2ex}

\noindent A ternary product $(f_1,f_2,f_3)$ is disjoint if the binary products, such as $(f_1,(f_2,f_3))$, one can build from it are disjoint. 
\end{defn}

\begin{rmk}{\em 
This notion of disjointness is not something that holds of binary products in an arbitrary category with products. For instance, in a lattice considered as a category, products are given by joins. In this setting, the product of a pair $x,y$ is disjoint iff the meet of $x,y$ is the least element of the lattice, i.e. if the pair $x,y$ is disjoint in the sense usually used in lattice theory. However, many categories have the property that all binary products are disjoint. This is the case in the category of non-empty sets, groups, rings, topological spaces, and so forth. 
}
\end{rmk}

The essential property in constructing an orthomodular structure from the direct product decompositions of an object in a category seems to be that a certain diagram built from a ternary direct product diagram forms a pushout. This was introduced in \cite{Harding5} under the name {\em honest category}.

\begin{defn}
A category is honest if it has finite products; all projections are epimorphisms; and for each ternary product \mbox{diagram} $(f_1,f_2,f_3)$, where $f_i:A\to A_i$, the following diagram is a pushout. 
\label{h}
\end{defn}
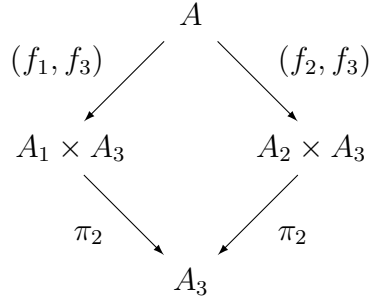
\vspace{2ex}

\begin{center}
\pushout{A}{A_1\times A_3}{A_2\times A_3}{A_3}{(f_1, f_3)}{(f_2, f_3)}{\pi_2}{\pi_2}{1}
\end{center}

In \cite{Harding5} it was shown that the disjoint binary decompositions of an object $A$ in an honest category form an \ts{oa}. We would like to show that an interval $[f_1,f_2]\!\downarrow$, where $f_i:A\to A_i$, is isomorphic to the \ts{oa} of disjoint binary decompositions of the factor $A_1$. For a disjoint decomposition of $A$ in the interval $[f_1,f_2]\!\downarrow$ we can produce a decomposition of $A_1$, but cannot show this decomposition is disjoint. However, by slightly modifying the construction of \cite{Harding5} we can obtain our result. 

\begin{defn}
A category is strongly honest if it has finite products; projections are epic; and for all ternary product diagrams $(f_1,f_2,f_3)$ the diagram of Definition~\ref{h} is a pushout. A category is very strongly honest if it is honest and all binary product diagrams are disjoint. 
\end{defn}

Clearly any strongly honest category is honest, and any very strongly honest category is strongly honest. The construction of an \ts{oa} for an honest category can therefore be applied to a strongly honest or very strongly honest category. However, we give a modified construction for strongly honest categories that uses all \mbox{binary} product decompositions, rather than just disjoint ones. This allows us to bypass the difficulty with disjointness when passing to a decomposition of a factor. We note that the two constructions will agree when applied to a very strongly honest category. 

\begin{thm}
Let $A$ be an object in a strongly honest category and $\mathcal{D}(A)$ be the collection of binary decompositions of $A$. Define a partial binary operation $\oplus$ on $\mathcal{D}(A)$ where $[f_1,f_2]\oplus [g_1,g_2]$ is defined if there is a ternary decomposition $[c_1,c_2,c_3]$ with 
\vspace{1ex}
\[[f_1,f_2] = [c_1,(c_2,c_3)]\quad\mbox{and}\quad[g_1,g_2]=[c_2,(c_1,,c_3)]\]

\vspace{1ex}
\noindent In this case, define $[f_1,f_2]\oplus[g_1,g_2]=[(c_1,c_2),c_3]$. Then, with this operation and constants $0=[\tau_A,1_A]$ and $1=[1_A,\tau_A]$, the decompositions $\mathcal{D}(A)$ form an \ts{oa}. 
\label{honest}
\end{thm}

\begin{proof}
The proof is nearly identical to that in \cite{Harding5} for honest categories. In \cite{Harding5}, we build an \ts{oa} from disjoint binary decompositions, and have a property that applies to disjoint ternary decompositions. In strongly honest categories, all steps of the proof, save one, carry through if we simply ignore considerations of disjointness. 

The exception is in proving that if $[f_1,f_2]\oplus[f_1,f_2]$ is defined, then $[f_1,f_2]$ is the zero of $\mathcal{D}(A)$, $[\tau_A,1_A]$. In \cite{Harding5} this was established using the disjointness of $[f_1,f_2]$. However, this is true of any binary decomposition in a strongly honest category. The proof of this given below uses two claims where $p:A\to P$ and $q:A\to Q$. 

\begin{claim}
If $(p,q)$ is a product diagram and $q$ is an isomorphism, then $P\simeq\Omega$. 
\label{w}
\end{claim}

\begin{proof}[Proof of Claim: ]
Note that $(p,1_A ,1_P ,p)$ is a pushout. Then as $\pi_1 :P\times\Omega\longrightarrow P$ and $\pi_1 :A\times\Omega\longrightarrow A$ are isomorphisms with $\pi_1 \circ (p,\tau_A ) = p$ and $\pi_1 \circ (1_A,\tau_A) = 1_A$, the diagram $((p,\tau_A) ,(1_A,\tau_A) ,\pi_1 ,p\circ\pi_1 )$ is a pushout (below left). If $q$ is an isomorphism, and $(p,q)$ is a product diagram, then $(p,1_A ,\tau_A )$ is also a product diagram. Strong honesty gives $((p,\tau_A),(1_A,\tau_A) ,\pi_2 ,\pi_2 )$ is a pushout (below right). Therefore $P\cong\Omega$.

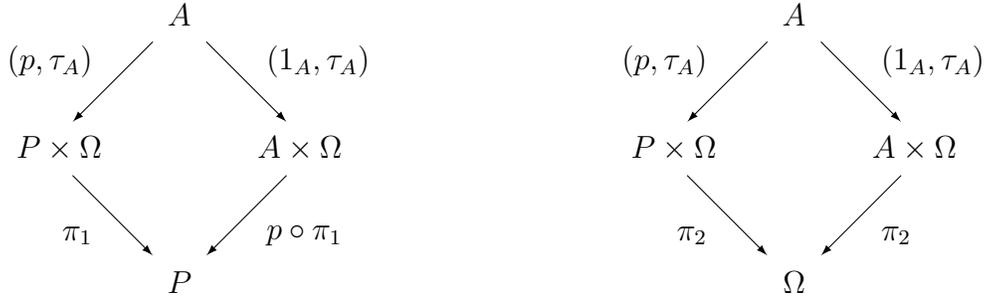
\vspace{1ex}
\begin{center}
\pushout{A}{P\times\Omega}{A\times\Omega}{P}{(p,\tau_A)}{(1_A,\tau_A)}{\pi_1}{p\circ\pi_1}{6} 
\ \hspace{1in} \ 
\pushout{A}{P\times\Omega}{A\times\Omega}{\Omega}{(p,\tau_A)}{(1_A,\tau_A)}{\pi_2}{\pi_2}{7} 
\end{center}
\vspace{-5ex}
\end{proof}

\begin{claim}
If $(p,p,q)$ is a product diagram, then $P\simeq\Omega$ and $q$ is an isomorphism.
\label{x}
\end{claim}
\vspace{-2ex}

\begin{proof}[Proof of Claim: ]
Strong honesty gives $((p,q),(p,q),\pi_2,\pi_2)$ is a pushout (below left). As $(p,p,q)$ is a product diagram and projections are epic, it follows that $(p,q)$ is epic, hence $((p,q),(p,q),1_{P\times Q},1_{P\times Q})$ is a pushout (below right). Therefore there is an isomorphism $k:P\times Q\to Q$ with $k\circ 1_{P\times Q}=\pi_2$. So $\pi_2:P\times Q\to Q$ is an isomorphism. Claim~\ref{w} then gives $P\simeq\Omega$. Thus $(\tau_A,\tau_A,q)$ is a product diagram, but so also is $(\tau_A,\tau_A,1_Q)$. From general properties of products, it follows that $q\simeq 1_Q$, hence $q$ is an isomorphism. 

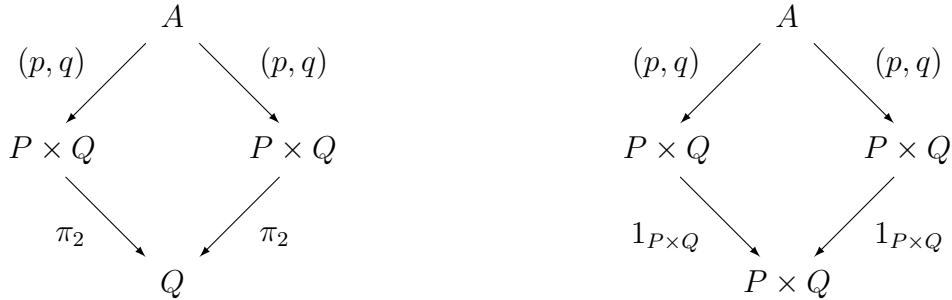
\vspace{1ex}
\begin{center}
\pushout{A}{P\times Q}{P\times Q}{Q}{(p, q)}{(p, q)}{\pi_2}{\pi_2}{8} 
\ \hspace{1in} \ 
\pushout{A}{P\times Q}{P\times Q}{P\times Q}{(p, q)}{(p, q)}{1_{P\times Q}}{1_{P\times Q}}{9} 
\end{center}

\vspace{-8ex}
\end{proof}

If $[f_1,f_2]\oplus [f_1,f_2]$ is defined there is a ternary decomposition $[c_1,c_2,c_3]$ of $A$ with $[c_1,(c_2,c_3)] = [f_1,f_2]$ and $[c_2,(c_1,c_3)]=[f_1,f_2]$. It follows that $(f_1,f_1,c_3)$ is a ternary decomposition of $A$. Claim~\ref{x} gives $F_1\simeq\Omega$, hence $f_1\simeq\tau_A$, and $c_3$ is an isomorphism, hence $c_3\simeq 1_A$. Thus $[f_1,f_2]=[\tau_A,1_A]$. 
\end{proof}

To conclude this section we note there are obviously many relationships between these constructions. For instance, the category of non-empty sets is strongly honest, and the construction of an orthostructure from a set $X$ given by Theorem~\ref{set} agrees with that given by Theorem~\ref{honest}. Similar comments hold for the many ways to create an orthostructure from a vector space. However, there is so far no unifying setting that includes all the results described in the above theorems, and their extensions discussed in the remarks. 

\section{The main result in the lattice and ring setting}

Here we prove our main result in the setting of bounded modular lattices and rings, and indicate extensions to other settings. The result in the setting of strongly honest categories is given in the following section. 

\begin{thm}
Let $L$ be a bounded modular lattice and $(a,b)$ be a complementary pair in $L$. Then the interval $(a,b)\!\downarrow$ of the \ts{omp} $L^{(2)}$ is isomorphic to the \ts{omp} $a\!\downarrow^{(2)}$. 
\label{latticeproof}
\end{thm}

\begin{proof}
Define maps $\Gamma:(a,b)\!\downarrow\,\,\to a\!\downarrow^{(2)}$ and $\Phi:a\!\downarrow^{(2)}\,\,\to (a,b)\!\downarrow$ as follows. 

\begin{eqnarray*}
\Gamma(x,y) &=& (x,y\wedge a)\\
\Phi(u,v)&=&(u,v\vee b)
\end{eqnarray*}
\vspace{0ex}

To see these are well defined, we must show they result in complementary pairs in the appropriate lattice. Suppose $(x,y)\in(a,b)\!\downarrow$. Then $x,y$ are complementary in $L$, $x\leq a$ and $b\leq y$. Then $x\wedge(y\wedge b)=0$, and as $x\leq a$ modularity gives $x\vee (y\wedge a) = (x\vee y)\wedge a=a$, so $x,y\wedge a$ are complements in $a\!\downarrow$. Conversely, if $u,v$ are complements in $a\!\downarrow$, then $u\vee(v\vee b)=a\vee b = 1$, and as $v\leq a$ modularity gives $u\wedge (v\vee b) \leq a\wedge (v\vee b) = v\vee (a\wedge b) = v$, so $u\wedge(v\vee b)\leq u\wedge v = 0$. 

So $\Gamma$ and $\Phi$ are well defined. It is obvious they preserve order. To see they are inverses of one another, note $\Phi\Gamma(x,y) = (x,(y\wedge a)\vee b)$ and $\Gamma\Phi(u,v) = (u,(v\vee b)\wedge a)$. As $b\leq y$ modularity gives $(y\wedge a)\vee b = y\wedge (a\vee b) = y$, and as $v\leq a$ modularity gives $(v\vee b)\wedge a = v\vee(a\wedge b) = v$. So $\Phi\Gamma$ and $\Gamma\Phi$ are identity maps. It remains only to show $\Gamma$ and $\Phi$ are compatible with orthocomplementations. We use $(x,y)^{\#}$ to denote orthocomplement in $(a,b)\!\downarrow$ and $'$ for orthocomplementation in both $L^{(2)}$ and $a\!\downarrow^{(2)}$. 

\begin{eqnarray*}
\Gamma((x,y)^{\#})&=&\Gamma((x,y)'\wedge(a,b))\\
&=&\Gamma(y\wedge a,x\vee b)\\
&=&(y\wedge a,(x\vee b)\wedge a)\\
&=&(y\wedge a,x)\\
&=&(\Phi(x,y))' \\
\end{eqnarray*}

\begin{eqnarray*}
\Phi((u,v)')&=&(v,u\vee b)\\
&=&((v\vee b)\wedge a,u\vee b)\\
&=&(u,v\vee b)'\wedge(a,b)\\
&=&(\Phi(u,v))^{\#}
\end{eqnarray*}
\end{proof} 

\begin{thm} 
For $e$ an idempotent of a ring $R$, let $R_e=\{x:ex=x=xe\}$. Then $R_e$ is a ring with unit $e$ under the multiplication and addition of $R$, and the interval $e\!\downarrow$ of the \ts{omp} $E(R)$ is equal to $E(R_e)$. 
\label{ringproof}
\end{thm}

\begin{proof} 
It is trivial that $R_e$ is a ring with unit $e$. Suppose $f$ belongs to the interval $e\!\downarrow$ of $E(R)$. Then $ef=f=fe$, so $f$ belongs to $R_e$ and is idempotent in $R_e$, so $f$ belongs to $E(R_e)$. Conversely, suppose $g$ belongs to $E(R_e)$. Then $g$ is idempotent in $R_e$, hence also in $R$, and $eg=g=ge$. So $g$ belongs to the interval $e\!\downarrow$ of $E(R)$. Thus as sets $e\!\downarrow$ is equal to $E(R_e)$. The definition of $\leq$ in both structures is $gh=g=hg$, so they coincide. It remains only to show their orthocomplementations agree. The orthocomplementation $\#$ in the interval $e\!\downarrow$ is given by $f^{\#}=f'\wedge e = (1-f)e$. Then as $fe=f$ this evaluates to $e-f$, which is the orthocomplementation in $E(R_e)$. 
\end{proof}

\begin{rmk}{\em 
Theorem~\ref{latticeproof} has several generalizations. First, one notices it applies to the case of symmetric lattices as each step only involves basic properties of $M$ and $M^*$-symmetry found in \cite{Maeda}. The main part is in showing the images of the isomorphisms are modular and dual modular pairs. Next, and perhaps somewhat surprising, one notices it applies to the relation algebra setting with the interval $(a,b)\!\!\downarrow$ in $R^{(2)}$ being isomorphic to $a\!\!\downarrow^{(2)}$ where $a\!\!\downarrow$ is naturally considered as a relation algebra. Here the key point is that small fragments of modularity hold in any relation algebra. These were discovered by Chin and Tarski, and their role in the current context is described in detail in \cite{Harding1}. Theorem~\ref{ringproof} also has generalizations to situations described in Remark~\ref{a}. 
}
\end{rmk}

\section{The main result in the categorical setting}

\begin{thm}
Suppose $A$ is an object in a strongly honest category $\mathcal{C}$ and $[h_1,h_2]$ is a binary decomposition of $A$ where $h_i:A\to H_i$. Then the interval $[h_1,h_2]\!\downarrow$ of the \ts{omp} $\mathcal{D}(A)$ is isomorphic to the \ts{omp} $\mathcal{D}(H_1)$. 
\end{thm}

\begin{proof}
We first define a map $\Gamma:[h_1,h_2]\!\downarrow\,\,\to \mathcal{D}(H_1)$. Suppose $[f_1,f_2]\leq[h_1,h_2]$. By definition of $\leq$ there is $[g_1,g_2]$ with $[f_1,f_2]\oplus[g_1,g_2]$ defined and equal to $[h_1,h_2]$, and as every \ts{oa} is cancellative this $[g_1,g_2]$ is unique. By definition of $\oplus$ there is a ternary decomposition $[c_1,c_2,c_3]$ of $A$ with $[f_1,f_2]=[c_1,(c_2,c_3)]$, $[g_1,g_2]=[c_2,(c_1,c_3)]$ and $[h_1,h_2]=[(c_1,c_2),c_3]$. So this ternary decomposition is equal to $[f_1,g_1,h_2]$. As $h_1\simeq (f_1,g_1)$ there is an isomorphism $\gamma:H_1\to F_1\times G_1$ with $\gamma\circ h_1 = (f_1,g_1)$, and this $\gamma$ is unique since $h_1$ is a projection and projections are epic. Basic properties of products show this $\gamma$ can be written $(\gamma_1,\gamma_2)$ where $[\gamma_1,\gamma_2]$ is a decomposition of $H_1$. We define $\Gamma([f_1,f_2])=[\gamma_1,\gamma_2]$. 

\setlength{\unitlength}{.05in}
\begin{center}
\begin{picture}(30,20)(0,0)
\put(0,7.5){\makebox(0,0)[r]{$A$}}
\put(25,15){\makebox(0,0)[c]{$H_1$}}
\put(25,0){\makebox(0,0)[c]{$F_1\times G_1$}}
\put(2,8){\vector(8,3){18}}
\put(2,6){\vector(8,-3){16}}
\put(25,12){\vector(0,-1){9}}
\put(12,13.5){\makebox(0,0)[br]{$h_1$}}
\put(10.5,1.5){\makebox(0,0)[tr]{$\small{(f_1,g_1)}$}}
\put(30,7.5){\makebox(0,0)[l]{$\gamma=(\gamma_1,\gamma_2)$}}
\end{picture}
\end{center}
\vspace{5ex}

The above discussion established the following.

\begin{claim}
If $[f_1,f_2] \oplus [g_1,g_2] = [h_1,h_2]$, then $\Gamma([f_1,f_2]) = [\gamma_1,\gamma_2]$ iff $\gamma_1\circ h_1 = f_1$ and $\gamma_2\circ h_1 = g_1$. 
\label{z}
\end{claim}

We next define $\Phi:\mathcal{D}(H_1)\to[h_1,h_2]\!\downarrow$. Let $[m_1,m_2]$ be a decomposition of $H_1$. Then $(m_1h_1,m_2h_1,h_2)$ is a ternary decomposition of $A$. Thus, by the definition of $\oplus$ we have $[m_1h_1,(m_2h_1,h_2)]\oplus [m_2h_1,(m_1h_1,h_2)] = [(m_1h_1,m_2h_1),h_2]$. Basic properties of products give $(m_1h_1,m_2h_1)\simeq h_1$, so this latter term is $[h_1,h_2]$. This shows $[m_1h_1,(m_2h_1,h_2)]\leq [h_1,h_2]$. We define $\Phi([m_1,m_2]) = [m_1h_1,(m_2h_1,h_2)]$. 

\begin{claim}
$\Gamma$ and $\Phi$ are mutually inverse bijections. 
\end{claim}

\begin{proof}[Proof of Claim: ] 
Suppose $[f_1,f_2]\in[h_1,h_2]\!\downarrow$ and $[g_1,g_2]$ is the decomposition with $[f_1,f_2]\oplus[g_1,g_2]=[h_1,h_2]$. Then $\Gamma([f_1,f_2]) = [\gamma_1,\gamma_2]$ where $(\gamma_1,\gamma_2)\circ h_1 = (f_1,g_1)$, hence $\gamma_1\circ h_1 = f_1$ and $\gamma_2\circ h_2=g_1$. Then $\Phi\Gamma([f_1,f_2]) = [\gamma_1h_1,(\gamma_2h_1,h_2)] = [f_1,(g_1,h_2)]$. In the discussion of the definition of $\Gamma$, we saw that $(g_1,h_2)=f_2$. It follows that $\Phi\circ\Gamma$ is the identity. For the other composite, suppose $[m_1,m_2]\in\mathcal{D}(H_1)$. Then $\Phi([m_1,m_2]) = [m_1h_1,(m_2h_1,h_2)]$ and in the discussion of the definition of $\Phi$ we saw that $[m_1h_1,(m_2h_1,h_2)]\oplus [m_2h_1,(m_1h_1,h_2)] = [h_1,h_2]$. So $\Gamma\Phi([m_1,m_2])$ is the unique isomorphism $(\gamma_1,\gamma_2)$ with $(\gamma_1,\gamma_2)\circ h_1=(m_1h_1,m_2h_1)$, which is $(m_1,m_2)$. So $\Gamma\circ\Phi$ is also the identity. 
\end{proof}

We say that a map $\Pi$ between \ts{oa}s preserves $\oplus$ if $x\oplus y$ being defined implies $\Pi(x)\oplus\Pi(y)$ is defined and $\Pi(x\oplus y)=\Pi(x)\oplus\Pi(y)$. 

\begin{claim}
$\Gamma$ preserves $\oplus$. 
\end{claim}

\begin{proof}
Note that the operation $\oplus$ in the \ts{oa} $[h_1,h_2]\!\!\downarrow$ is the restriction of the operation $\oplus$ of $\mathcal{D}(A)$.
Suppose $[e_1,e_2]$ and $[f_1,f_2]$ belong to $[h_1,h_2]\!\downarrow$ and $[e_1,e_2]\oplus [f_1,f_2]$ is defined. Then there is $[g_1,g_2]$ with $([e_1,e_2]\oplus[f_1,f_2])\oplus[g_1,g_2] = [h_1,h_2]$. Let $(c_1,c_2,c_3)$ and $(d_1,d_2,d_3)$ be the ternary decompositions of $A$ realizing $[e_1,e_2]\oplus [f_1,f_2]$ and $([e_1,e_2]\oplus[f_1,f_2])\oplus[g_1,g_2]$ respectively. Then 

\begin{eqnarray*}
\mbox{$[c_1,(c_2,c_3)]$}&=&\mbox{$[e_1,e_2]$}\\
\mbox{$[c_2,(c_1,c_3)]$}&=&\mbox{$[f_1,f_2]$}\\
\mbox{$[(c_1,c_2),c_3]$}&=&\mbox{$[e_1,e_2]\oplus[f_1,f_2]$}\\
&&\\
\mbox{$[d_1,(d_2,d_3)]$}&=&\mbox{$[e_1,e_2]\oplus[f_1,f_2]$}\\
\mbox{$[d_2,(d_1,d_3)]$}&=&\mbox{$[g_1,g_2]$}\\
\mbox{$[(d_1,d_2),d_3]$}&=&\mbox{$[h_1,h_2]$}
\end{eqnarray*}
\vspace{0ex}

\noindent It follows that $c_1\simeq e_1$, $c_2\simeq f_1$, $c_3\simeq(d_2,d_3)$, $d_1\simeq(c_1,c_2)$, $d_2\simeq g_1$ and \mbox{$d_3\simeq h_2$}. As $[d_1,d_2,d_3] = [(e_1,f_1),g_1,h_2]$ it follows from basic properties of products that $[e_1,f_1,g_1,h_2]$ is a decomposition of $A$ and $[(e_1,f_1,g_1),h_2] = [(d_1,d_2),d_3] = [h_1,h_2]$. Thus $(e_1,f_1,g_1)\simeq h_1$, so there is an isomorphism $\gamma:H_1\to E_1\times F_1\times G_1$ with $\gamma\circ h_1 = (e_1,f_1,g_1)$. Basic properties of products show $\gamma=(\gamma_1,\gamma_2,\gamma_3)$ where $[\gamma_1,\gamma_2,\gamma_3]$ is a ternary decomposition of $H_1$ and $\gamma_1h_1=e_1$, $\gamma_2h_1=f_1$ and $\gamma_3h_1=g_1$. 

As $([e_1,e_2]\oplus[f_1,f_2])\oplus [g_1,g_2]$ is defined and equal to $[h_1,h_2]$, the associativity condition in any \ts{oa} implies $[f_1,f_2]\oplus[g_1,g_2]$ is defined and $[e_1,e_2]\oplus([f_1,f_2]\oplus[g_1,g_2])$ is defined and is equal to $[h_1,h_2]$. As $[e_1,f_1,g_1,h_2]$ is a decomposition of $A$, so is $[f_1,g_1,(e_1,h_2)]$, and this ternary decomposition realizes $[f_1,f_2]\oplus[g_1,g_2]$ being defined, hence being equal to $[(f_1,g_1),(e_1,h_2)]$. Thus $[e_1,e_2]\oplus[(f_1,g_1),(e_1,h_2)]=[h_1,h_2]$. Since $\gamma_1h_1 = e_1$ and $(\gamma_2,\gamma_3) h_1 = (f_1,g_1)$, Claim~\ref{z} shows $\Gamma([e_1,e_2]) = (\gamma_1,(\gamma_2,\gamma_3))$.  

A similar calculation making use of the commutativity of $\oplus$ in any \ts{oa} shows that $\Gamma([f_1,f_2])=(\gamma_2,(\gamma_1,\gamma_3))$, and another shows $\Gamma([e_1,e_2]\oplus[f_1,f_2]) = [(\gamma_1,\gamma_2),\gamma_3]$. Thus $(\gamma_1,\gamma_2,\gamma_3)$ realizes $\Gamma([e_1,e_2])\oplus \Gamma([f_1,f_2])$ is defined and equal to $\Gamma([e_1,e_2]\oplus[f_1,f_2])$. This concludes the proof of the claim. 
\end{proof}

\begin{claim}
$\Phi$ preserves $\oplus$.
\end{claim}

\begin{proof}[Proof of Claim: ]
Suppose that $[m_1,m_2]$ and $[n_1,n_2]$ are decompositions of $H_1$ with $[m_1,m_2]\oplus [n_1,n_2]$ defined and $[p_1,p_2,p_3]$ is a ternary decomposition realizing this. So $[m_1,m_2]=[p_1,(p_2,p_3)]$, $[n_1,n_2]=[p_2,(p_1,p_3)]$ and $[m_1,m_2]\oplus[n_1,n_2]=[(p_1,p_2),p_3]$. 
The definition of $\Phi$ gives 

\begin{eqnarray*}
\Phi([m_1,m_2])&=&[m_1h_1,(m_2h_1,h_2)]\\
\Phi([n_1,n_2])&=&[n_1h_1,(n_2h_1,h_2)]\\
\Phi([m_1,m_2]\oplus[n_1,n_2])&=&[(m_1,n_1)h_1,(p_3h_1,h_2)]
\end{eqnarray*}
\vspace{0ex}

\noindent As $[p_1,p_2,p_3]$ is a decomposition of $H_1$ and $[h_1,h_2]$ is a decomposition of $A$, general properties of products show that $[p_1h_1,p_2h_1,p_3h_1,h_2]$ is a decomposition of $A$. 

Consider the ternary decomposition $[p_1h_1,p_2h_1,(p_3h_1,h_2)]$ of $A$. We note that $p_1h_1\simeq m_1h_1$, $(p_2h_1,(p_3h_1,h_2))\simeq ((p_2,p_3)h_1,h_2) \simeq (m_2h_1,h_2)$ and similarly that $p_2h_1\simeq n_1h_1$, and $(p_1h_1,(p_3h_1,h_2))\simeq (n_2h_1,h_2)$. This shows $\Phi([m_1,m_2])\oplus\Phi([n_1,n_2])$ is defined and equal to $[(p_1h_1,p_2h_1),(p_3h_1,h_2)]$. As $(p_1h_1,p_2h_1)\simeq (m_1,n_1)h_1$, it follows that $\Phi([m_1,m_2])\oplus\Phi([n_1,n_2]) = \Phi([m_1,m_2]\oplus[n_1,n_2])$. 
\end{proof}

To show $\Gamma$ and $\Phi$ are mutually inverse \ts{oa} isomorphisms, it remains only to show they preserve bounds. As we know they are inverses, it suffice to show one of them preserves bounds. In the \ts{oa} $[h_1,h_2]\!\downarrow$ we have $0=[\tau_A,1_A]$ and $1=[1_A,\tau_A]$. Note $[\tau_A,1_A]\oplus[h_1,h_2]=[h_1,h_2]$, and it follows from Claim~\ref{z} that $\Gamma([\tau_A,1_A]) = [\tau_{H_1},1_{H_1}]$ and $\Gamma([h_1,h_2]) = [1_{H_1},\tau_{H_1}]$. So $\Gamma$ preserves bounds. This concludes the proof of the theorem. 
\end{proof}

\begin{rmk}{\em 
We do not know if this result in the strongly honest setting extends to the honest setting. As mentioned above, the difficulty is in establishing disjointness of $\Gamma([f_1,f_2])$ etc. 
}
\end{rmk}

\end{document}